%% file: preprint.tex
\documentclass[11pt,reqno,a4paper]{amsart}

\usepackage[utf8]{inputenc}
\usepackage[british]{babel}

\usepackage{amsmath,amssymb,amsthm,amsaddr}
\usepackage{enumitem}
\usepackage{mathtools}

\usepackage{tikz}
\usepackage[framemethod=TikZ]{mdframed}

\usepackage{multirow}
\usepackage{colortbl}

\usepackage{caption}
\usepackage{subcaption}

\usepackage[margin=3cm,top=4cm,bottom=3.5cm,footskip=2cm,headsep=2cm]{geometry}

\newtheorem{theorem}{Theorem}%[section]

\theoremstyle{plain}

\newtheorem{proposition}[theorem]{Proposition}

\theoremstyle{definition}

\newtheorem{remark}[theorem]{Remark}
\newtheorem{assumption}[theorem]{Assumption}

\usepackage[bookmarksnumbered=true,draft=false]{hyperref}
\hypersetup{
  colorlinks=true,
  allcolors = red,
  linkcolor=blue, 
  menucolor=blue,
  urlcolor = blue,
  frenchlinks=false,
  pdfborder={0 0 0},
  naturalnames=false,
  hypertexnames=false,
  breaklinks
}

\input{abbreviations}

\title[]{A kinetic chemotaxis model\\ and its diffusion limit in slab geometry}
\author{H. Egger$^{a,b}$, K. Hellmuth$^c$, N. Philippi$^b$, \and M. Schlottbom$^d$}
\address{$^a$Institute for Numerical Mathematics, Johannes-Kepler University Linz, Austria\\
$^b$Johann Radon Institute for Computational and Applied Mathematics, Linz, Austria\\
$^c$Department of Mathematics, University of Würzburg, Germany\\
$^d$Department of Applied Mathematics, University of Twente, Enschede, The Netherlands}
\email{herbert.egger@jku.at}
\email{kathrin.hellmuth@uni-wuerzburg.de}
\email{nora.philippi@ricam.oeaw.ac.at}
\email{m.schlottbom@utwente.nl}

\begin{document}

\begin{abstract}
Chemotaxis describes the intricate interplay of cellular motion in response to a chemical signal. We here consider the case of slab geometry which models chemotactic motion between two infinite membranes. Like previous works, we are particularly interested in the asymptotic regime of high tumbling rates.
We establish local existence and uniqueness of solutions to the kinetic equation and show their convergence towards solutions of a parabolic Keller-Segel model in the asymptotic limit. 
In addition, we prove convergence rates with respect to the asymptotic parameter under additional regularity assumptions on the problem data. 
Particular difficulties in our analysis are caused by vanishing velocities in the kinetic model as well as the occurrence of boundary terms. 
\end{abstract}

\maketitle

\vspace*{-1em}

\begin{quote}
\end{quote}

\begin{quote}
\noindent
{\small {\bf AMS-classification (2000):}  
92C17, % chemotaxis
35B40, % Asymptotic behaviour of PDEs
% 35Q49, % Transport equations
% 35A01, % Existence problems for PDEs
% 35B45, % A-priori estimates for solutions of PDEs
35Q92, % PDEs in connection with biology
35M33  % IBVPs fo mixed type
}
\end{quote}

% \begin{quote}
% \noindent
% {\small {\bf Funding:}  
% K.H. acknowledges support of the German Academic Scholarship Foundation (Studienstiftung des deutschen Volkes) as well as the Marianne-Plehn-Programm.
% }
% \end{quote}

\section{Introduction}
\label{sec:intro}

% \begingroup

The term \emph{chemotaxis} is generally used to describe the motion of cells or bacteria in response to the gradient of a chemical stimulus~\cite{Kretschmer1980,Perthame07}. It has been recognized as a fundamental mechanism in applications like immune response~\cite{Wu05}, embryological development~\cite{Dormann06}, or bacterial population dynamics~\cite{lauffenburger91}. 
%
% The importance of spatial constraints to the movement has been highlighted in \cite{ Bhattacharjee19, moure18}. 
%
Starting from stochastic descriptions, the first macroscopic models for chemotaxis describing the evolution of the particle densities were derived by Patlak~\cite{Patlak1953}, Keller and Segel~\cite{KellerSegel1971}. We refer to~\cite{HillenPainter09,Horstmann03} and~\cite{Perthame07} for a survey of results and additional references. 
Alt and co-workes~\cite{Alt1980,Othmer88} formulated multi-dimensional kinetic models, describing the particle densities, and investigated their diffusion limits by an asymptotic analysis. 
The rigorous justification of the diffusion limit was put forward in \cite{HillenStevens00} for the one-dimensional case, and in \cite{OthmerHillen2002} and \cite{Chalub04} for multiple space dimensions.

\subsection*{Scope.}
In this paper, we study a kinetic model for chemotaxis and its diffusion limit in \emph{slab geometry}, which is of relevance for the propagation of bacteria when passing membranes or small slabs of porous media, soil or biological gels \cite{bhattacharjee21,Gaveau17,SadrGhayeni1999}. Numerical studies for chemotaxis in slab geometry can be found in \cite{CarilloYan13,Gosse13}. 
Corresponding diffusion limits have been considered in  
\cite{BalRyzhik00} in the context of radiative transfer. 

Let us note that the particle density in slab geometry is a function of one space and one velocity variable which, in contrast to \cite{HillenStevens00,KellerSegel1971,Patlak1953}, varies continuously between $\pm 1$. The kinetic equation hence does not degenerate into a system of two coupled hyperbolic equations.

\subsection*{Main results.}
The focus of this paper lies on a rigorous asymptotic analysis for chemotaxis in slab geometry. In particular, we establish
\begin{itemize}
\item local well-posedness of a kinetic model for chemotaxis in slab geometry and a-priori estimates for solutions which are explicit in the asymptotic parameter;
\item convergence to solutions of a one-dimensional Keller-Segel system in the diffusion limit of vanishing asymptotic parameter and quantitative convergence rates. 
\end{itemize}
Similar to previous work, we use energy-estimates and fixed-point arguments to establish existence and uniqueness of solutions to the kinetic model and to prove convergence to solutions of the limit system. 
To the best of our knowledge, our results can, however, not be deduced from previous work: The assumptions in \cite{Chalub04} do not allow to treat scattering operators of the form considered here and the work \cite{BalRyzhik00} does not include the nonlinear coupling to the equation for the chemoattractant. The results of \cite{HillenStevens00} are valid for a  discrete velocity model, but do not generalize to the case of slab geometry with continuous velocities.
As mentioned in \cite{Perthame07}, the asymptotic analysis for the kinetic chemotaxis model is in line with that for the radiative transfer equation. We thus follow a similar approach as outlined in \cite{Bardos88,Dautray93,EggerSchlottbom14}
and derive suitable extensions to chemotaxis in slab geometry. 

\subsection*{Outline}
In Section~\ref{sec:prelim}, we fix  our notation and state the kinetic model under consideration. We then introduce our assumptions and state the main results of the paper. 
Their proofs are give in Sections~\ref{sec:proof:thm:2}--\ref{sec:proof:thm:3b}. 
In Section~\ref{sec:conclusion}, we review some details in the proofs and discuss some possible extensions of our results.

%%%%%%%%%%%%%%%%%%%%%%%%%%%%%%%%%%%%%%%%%
\section{Preliminaries and main results}
\label{sec:prelim}
%%%%%%%%%%%%%%%%%%%%%%%%%%%%%%%%%%%%%%%%%

We start with introducing our notation and the model problem under consideration, then state the main assumptions for the subsequent analysis and present our main results. 

\subsection{Notation}
Let $\X=(0,\ell)$, $\ell > 0$, denote the spatial domain. We write $L^p(\X)$, $W^{k,p}(\X)$, $H^k(\X)$ for the usual Sobolev spaces defined on $\X$. Similar notation is used for functions defined on other domains. 
The range of velocities is denoted by $\V = (-1,1)$ and we write $\Q=\X \times \V$ for the phase space. 
Note that any function $\bar f \in L^p(\X)$ can be interpreted as a function $f \in L^p(\Q)$ by constant extension $f(x,v) = \bar f(x)$. 
We use the bar symbol throughout to indicate that functions do not depend on velocity $v$. The same symbol is used for the velocity average of a function $f\in L^p(\Q)$, i.e., 
\begin{align*}
    \bar f = \frac{1}{2}\int_\V f(v)\ dv.
\end{align*}
Following \cite{Dautray93}, we define in- and outflow boundaries of the phase space $\Q$ by 
\begin{align*}
    \Gamma^{in} =\{0\}\times (0,1)\cup\{\ell\}\times (-1,0) \qquad  \text{and} \qquad
    \Gamma^{out} = \{0\}\times (-1,0)\cup \{\ell\}\times (0,1);
\end{align*}
and we set $\Gamma=\Gamma^{in} \cup \Gamma^{out}$; see Figure~\ref{fig:inflowoutflowphase} for an illustration. 
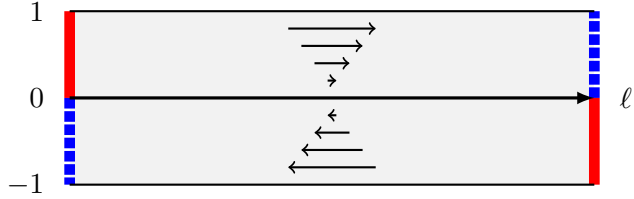
\begin{figure}[ht!]
\centering
\input{domain.tikz}
\caption{Phase space $\Q=\X \times \V$ and in-/outflow boundaries $\Gamma^{in}$, $\Gamma^{out}$ depicted in (red, solid) and (blue, dotted), respectively.}
\label{fig:inflowoutflowphase}
\end{figure}
By the product rule for differentiation and the fundamental theorem of calculus, we see that 
\begin{align}
\int_{\Q} &v\dx u(x,v) \,  w(x,v) +  u(x,v) \, v\dx w(x,v) \, d(x,v) 
= \int_{\Gamma} v \,u(x,v) \, w(x,v) \, d\Gamma \notag \\
&= \int_{\Gamma^{out}} u(x,v) \, w(x,v) \, |v| \, d\Gamma 
- \int_{\Gamma^{in}} u(x,v) \, w(x,v) \, |v| \, d\Gamma \label{eq:ibp}
\end{align}
for sufficiently smooth functions $u,w$ defined on $\Q$. 
We write $L^p(\Gamma;|v| d\Gamma)$ for the space of functions defined on $\Gamma$ whose $p$-th power is integrable with respect to the measure $|v| d\Gamma$. 
Given $T>0$, we denote by $L^p(0,T;Y)$ and $W^{k,p}(0,T;Y)$ the Bochner spaces of functions defined on $(0,T)$ with values in some Banach space $Y$ equipped with the usual norms; see \cite{Evans10} for instance. 
For brevity, we sometimes write $L^p(Y)$ for  $L^p(0,T;Y)$, when the meaning is clear from the context.
The spaces $C^k([0,T];Y)$ consists of all functions which are $k$-times continuously differentiable with respect to time.

\subsection{Model problem and main assumptions}
The differential equations governing the kinetic chemotaxis model to be considered in the rest of the manuscript are given by
\begin{alignat}{2}
    \eps^2 \dt u^\eps + \eps v \dx u^\eps + \sigma(u^\eps-\bar u^\eps) &= \eps\alpha v \dx \ceps \bar u^\eps \qquad &\text{in } \Q\times (0,T),\label{eq:u}\\
    \dt\bar c^\eps - D\dxx \ceps + \beta \ceps &= \gamma\bar u^\eps &\text{in } \X\times (0,T).\label{eq:c}
\end{alignat}
Recall that $\bar u^\eps=\frac{1}{2} \int_\V u^\eps(x,v) \, dv$ denotes the velocity average of the population density~$u^\eps$, whereas the concentration $\ceps$ does not depend on $v$ naturally.
The two equations are complemented by the boundary and initial conditions 
\begin{alignat}{8}
    % u^\eps & = \bar g_u \qquad & \text{ on }  \Gamma^{in} \times (0,T), \label{eq:u:bc}\\
    % \ceps & = \bar g_c & \text{on } \partial \X \times (0,T).\label{eq:c:bc}
    u^\eps &= \bar g_u \quad && \text{on }  \Gamma^{in} \times (0,T), \qquad \qquad &
    u^\eps(0) &= \bar u_0 \quad &&\text{on }\Q, \label{eq:u:ibc} \\
    \ceps & = \bar g_c \quad && \text{on } \partial \X \times (0,T), \qquad & 
    \ceps(0) &= \bar c_0 \quad &&\text{on } \X.\label{eq:c:ibc}
\end{alignat}
The boundary and initial data are thus chosen independent of the velocity $v$. This facilitates the analysis in the asymptotic regime but could be relaxed to some extent.
For our arguments, we will consider different time horizons $0 < T \le T^*$ and asymptotic parameters $0 < \eps \le \eps^*$ with $T^*,\eps^*>0$ fixed, and we make use of the following assumptions which allow us to establish the existence of sufficiently regular solutions to the model \eqref{eq:u}--\eqref{eq:c:ibc}.
\begin{assumption}\label{ass:1}
The parameters $\alpha,  \beta, \gamma\ge 0$ and $D,\eps^*,T^*>0$ are given constants. 
Furthermore $\sigma\in L^\infty(\X)$ with $0<\sigma_{min}\le\sigma(x)\le\sigma_{max}$ for a.a. $x \in \X$.
The initial data satisfy $\bar u_0 \in H^2(\X)$, $\bar c_0 \in H^4(\X)$, 
and the two boundary data $\bar g_u, \bar g_c \in W^{2,\infty}(0,T^*;\RR^2)$. 
We use the same notation $\bar g_u,\bar g_c$ for the linear extensions of the boundary data  to function $\bar g_u,\bar g_c \in W^{2,\infty}(0,T^*;W^{2,\infty}( \X))$.
In addition, we require the compatibility conditions $\bar g_u(0)=\bar u_0$, $\bar g_c(0)=\bar c_0$, and $\dt \bar g_c(0)=D \dxx \bar c_0 - \beta \bar c_0 + \gamma \bar u_0$ hold for $x \in \partial\X$.
\end{assumption}

%
%\kh{we need $\bar g_u\in H^2_t(H^2_x)$: Proof Prop 7 $\dtt \bar g_u\in L^2(L^2)$}

\subsection{Main results}
The above conditions allow us to establish the existence of a unique regular solution of the problem and to derive corresponding a-priori bounds. Their explicit dependence on the parameter $\eps$ will be of importance for the asymptotic analysis later on. 
\begin{theorem}\label{thm:2}
Let Assumption~\ref{ass:1} hold. Then, there exists a time horizon $T>0$ such that for any $0<\eps\le \eps^*$ the system \eqref{eq:u}--\eqref{eq:c:ibc} has a unique solution 
\begin{align*}
&u^\eps \in C^1([0,T];L^2(\Q)) \qquad \text{with} \qquad v \dx u^\eps \in C^0([0,T];L^2(Q)),\\
&
\ceps 
\in H^2(0,T;L^2(\X))\cap H^1(0,T;H^2(\X)).
\end{align*}
Moreover, the following a-priori bounds hold with a uniform constant $C>0$:
\begin{alignat*}{8}
&(a) \quad && \|u^\eps\|_{C^0([0,T];L^2(\Q))} \le C; 
\qquad &
&(b) \quad && \|u^\eps-\bar u^\eps\|_{L^2(0,T;L^2(\Q))} \le C\eps; 
\\
&(c) \quad && \|u^\eps - \bar g_u\|_{L^2(0,T;L^2(\Gamma^{out};|v|d\Gamma))} \le C \sqrt{\eps}; 
\qquad  &
&(d) \quad && \|\eps \dt u^\eps\|_{C^0([0,T];L^2(\Q))} \le C; 
\\
&(e) \quad && \|v\dx u^\eps\|_{L^2(0,T;L^2(\Q))} \le C; 
\qquad &
&(f) \quad && \|\ceps\|_{H^1(0,T;L^2(\X))}\le C; 
\\
&(g) \quad && \|\ceps\|_{L^2(0,T;H^2(\X))} \le C;
\qquad &
&(h) \quad && \| \eps \dt \ceps \|_{L^4(0,T;W^{1,\infty}(\X))} \le C. 
\end{alignat*}
%\textcolor{purple}{CHECK ESTIAMTE (h)}
% \begin{enumerate}[label={(\alph*)}]\setlength{\itemsep}{4pt}
% %
% \item 
% $\|u^\eps\|_{C^0([0,T];L^2(\Q))} \le C$,
% \,
% $\|\eps \dt u^\eps\|_{C^0([0,T];L^2(\Q))} \le C$, 
% \, 
% $\|u^\eps-\bar u^\eps\|_{L^2(0,T;L^2(\Q))} \le C\eps$; \label{a}
% \item 
% $\|u^\eps\|_{L^2(\Gamma^{out};|v|ds)}  \le C \sqrt{\eps}$; 
% $\|u^\eps - \bar g_u\|_{L^2(\Gamma^{in};|v|ds)} \le C \sqrt{\eps}$, \, 
% $\|v\dx u^\eps\|_{L^2(0,T;L^2(\Q))} \le C$;\label{b}
% %
% \item  
% $\|\ceps\|_{H^1(0,T;L^2(\X))}\le C$, \,  $\|\ceps\|_{L^2(0,T;H^2(\X))} \le C$, \, and \,
%  $\| \eps \dt \ceps \|_{L^4(0,T;L^\infty(\X))} \le C$. \label{c}
% \end{enumerate}
\end{theorem}
%
%\subsection*{Outline of the proof.}
The existence of a unique solution will be established via fixed-point arguments and results about the linearized evolution problems corresponding to the two differential equations. %
The bounds for the solution are derived by careful energy estimates and interpolation arguments. The detailed proof is presented in Section~\ref{sec:proof:thm:2}. 

\bigskip 

The second main result of the paper is concerned with the behavior of solutions in the asymptotic limit $\eps \to 0$. Our main findings can be summarized as follows. 
\begin{theorem}\label{thm:3}
Let Assumption~\ref{ass:1} hold and $(u^\eps,\ceps)$ be the solutions of \eqref{eq:u}--\eqref{eq:c:ibc} provided by Theorem~\ref{thm:2} for a sequence $\eps \to 0$.
Then $u^\eps \rightharpoonup \bar u^0$ weakly in $L^2(0,T;L^2(\Q))$ and $\ceps \to \co$ in $L^2(0,T;H^1(\X))$, where $(\bar u^0,\co)$ is the unique weak solution of the Keller-Segel system 
\begin{alignat}{2}
\dt\bar u^0 - \dx(\bar \mu\dx\bar u^0 - \bar\chi\dx\co\bar u^0) &= 0 \qquad &&\text{in } \X \times (0,T), \label{eq:u0}\\
\dt\co - D\dxx\co + \beta\co &= \gamma\bar u^0 \qquad &&\text{in } \X\times(0,T), \label{eq:c0}
\end{alignat}
with 
motility and chemotactic coefficient 
\begin{align}
\bar \mu(x) = \sigma(x)^{-1} \frac{1}{2}\int_\V v^2\ dv \qquad \text{and} \qquad \bar\chi(x) = \alpha\bar \mu(x), 
\end{align}
and with boundary and initial conditions
\begin{alignat}{8}
\bar u^0 &= \bar g_u \quad && \text{on } \partial\X\times(0,T), \qquad &
\bar u^0(0) &= \bar u_0 \quad && \text{on } \X, \label{eq:u0:ibc}\\ 
\co &= \bar g_c \quad &&\text{on } \partial \X\times(0,T), 
\qquad &
\bar c^0(0) &= \bar c_0 \quad &&\text{on } \X. \label{eq:c0:ibc}
\end{alignat}
If, in addition, $\sigma \in W^{1,\infty}(\X)$ and $\bar u_0 \in H^2(\X)$, then the asymptotic error is bounded by  
\begin{align} \label{eq:asym}
   \|\uepsg-
   \uogbar\|_{L^\infty(0,T;L^2(\Q))} + \|\cepsg - \cog\|_{L^\infty(0,T;H^1(\X))} \le C \eps
\end{align}
with a constant $C$ that can be chosen independent of $\eps$.
\end{theorem}
%
%\subsection*{Outline of the proof.}
The verification of the above assertions relies on the a-priori bounds of Theorem~\ref{thm:2}, the extraction of convergent subsequences, a weak characterization of solutions to the limit system, and an extension of the arguments used in \cite{EggerSchlottbom12,EggerSchlottbom14_Lpestimates} for the analysis of the radiative transfer equation.  
%
%Let us note that, as a by-product of our analysis, we also obtain the existence of a weak solution to the nonlinear limit problem \eqref{eq:u0}--\eqref{eq:c0:ibc} on the time interval $[0,T]$. 
%
The detailed proof of Theorem~\ref{thm:3} is presented in Sections~\ref{sec:proof:thm:3a} and \ref{sec:proof:thm:3b}.

%%%%%%%%%%%%%%%%%%%%%%%%%%%%%%%%%%%%%%%%%%%%%%%%
\section{Proof of Theorem~\ref{thm:2}}
\label{sec:proof:thm:2}
%%%%%%%%%%%%%%%%%%%%%%%%%%%%%%%%%%%%%%%%%%%%%%%%

The existence of a unique solution to \eqref{eq:u}--\eqref{eq:c:ibc} will be proven by constructing a fixed-point of the mapping $\Phi : z \mapsto u$, where $(u,\bar c)$ solves the linearized equations 
\begin{alignat}{2}
\eps^2 \dt u + \eps v \dx u + \sigma (u - \bar u) &= \eps \alpha v \dx \bar c \, \bar z \qquad &&\text{in } \Q \times (0,T), \label{eq:u:lin}\\
\dt \bar c - D \dxx \bar c + \beta \bar c &= \gamma \bar z \qquad &&\text{in } \X \times (0,T), \label{eq:c:lin}
\end{alignat}
together with the boundary and initial conditions \eqref{eq:u:ibc}--\eqref{eq:c:ibc}, now required for $u$ and $\bar c$ instead of $u^\eps$ and $\ceps$.
Also recall that $\bar z = \frac{1}{2} \int_\V z(\cdot,v) \, dv$ denotes the velocity average.

Before we are able to prove the assertions of Theorem~\ref{thm:2}, we first collect some preliminary results about the two linear problems \eqref{eq:u:lin}--\eqref{eq:c:lin} defining $\bar c$ and $u$.
Throughout our proofs, we denote by $C_i$ and $C_i'$ constants that only depend on the problem data, in particular, the upper bounds $\eps^*$, $T^*$ for the asymptotic parameter and the time horizon. 

\subsection{Linearized diffusion equation}

By application of standard results for parabolic differential equations, we obtain the following statements. 
\begin{proposition} \label{pro:c}
Let Assumption~\ref{ass:1} hold. Then, for $z \in C^1([0,T];L^2(Q))$ with $z(0)=\bar u_0$, equation \eqref{eq:c:lin} has a unique solution $\bar c$, additionally satisfying the initial and boundary conditions $\bar c(0)=\bar c_0$ on $\X$ and $\bar c = \bar g_c$ on $\partial\X \times (0,T)$. Moreover, for any $p<\infty$,
\begin{align*}
\|\bar c\|_{W^{1,p}(0,T;L^2(\X))} + \|\bar c\|_{L^p(0,T;H^2(\X))} &\le C_1 + C_2 \|\bar z\|_{L^p(0,T;L^2(\X))} \\
\|\dt \bar c\|_{W^{1,p}(0,T;L^2(\X))} + \|\dt \bar c\|_{L^{p}(0,T;H^2(\X))} &\le C_1 + C_2 \|\dt \bar z\|_{L^p(0,T;L^2(\X))} 
\end{align*}
with $C_1,C_2$ depending only on $p$ and the bounds for the problem data in Assumption~\ref{ass:1}.
\end{proposition}
\begin{proof}
The function $\bar d =  \bar c - \bar g_c$ satisfies 
\begin{alignat}{2}
\dt \bar d - D \dxx \bar d + \beta \bar d &= \bar f  \qquad && \text{in } \X \times (0,T), \label{eq:lin:c1}\\
\bar d(0) &= \bar d_0 \qquad && \text{on } \X, \label{eq:lin:c2}
\end{alignat}
with $\bar f = \gamma \bar z - \dt \bar g_c - \beta \bar g_c$, $\bar d_0 = \bar c_0 - \bar g_c(0)$, and  
$\bar d=0$ on $\partial\X \times (0,T)$. 
By the conditions of Assumption~\ref{ass:1}, we can verify that $\bar f \in L^p(0,T;L^2(\X))$ and $\bar d_0 \in H_0^1(\X)\cap H^2(\X)$. 
Existence of a unique solution $\bar c \in W^{1,p}(0,T;L^2(\X)) \cap L^p(0,T;H^2(\X))$ and the first bound then follow from known maximal parabolic regularity of the heat equation; see \cite{Augner2024,Dore93}.
% \kh{I found this \cite{Augner2024} an alternative reference for Maximal regularity of the heat equation - it has a very direct statement of $L^p-L^q$ MaxReg for the Dirichlet BVP for the heat equation in the introduction. }
%see \cite[Chapter~7, Thm.~3 and 5]{Evans10}. 
%
By formally differentiating \eqref{eq:c:lin} in time, we see that $\bar d =  \dt \bar c - \dt \bar g_c$ also satisfies \eqref{eq:lin:c1}--\eqref{eq:lin:c2}, now with $\bar f = \gamma \dt \bar z - \dtt \bar g_c - \beta \dt \bar g_c$ and $\bar d_0= D \dxx \bar c_0 - \beta \bar c_0 + \gamma\bar u_0 - \dt \bar g_c(0)$. 
By Assumption~\ref{ass:1} and the regularity of $\bar z$, we can bound $\bar d_0 \in H^1_0(\X)\cap H^2(\X)$ and $\bar f \in L^p(0,T;L^2(\X))$, which leads to the estimate for the time derivatives. 
\end{proof}

\subsection{Linearized kinetic equation}

As a next step, we establish existence of a unique solution to~\eqref{eq:u:lin} with appropriate boundary and initial conditions, and $\bar c$, $\bar z$ assumed given.
\begin{proposition} \label{pro:u}
Let Assumption~\ref{ass:1} hold. Then for any $z \in C^1([0,T];L^2(\Q))$ and $\bar c$ as given by Proposition~\ref{pro:c}, equation \eqref{eq:u:lin} has a unique solution $u \in C^1([0,T];L^2(\Q))$ with derivative $v \dx u \in C^0([0,T];L^2(\Q))$, satisfying $u(0)=\bar u_0$ on $\Q$ and $u=\bar g_u$ on $\Gamma^{in} \times (0,T)$.
\end{proposition}
\begin{proof}
Similar to before, we define $w =  u - \bar g_u$, which can be seen to satisfy 
\begin{alignat}{2}
\eps^2 \dt w + \eps v \dx w + \sigma (w - \bar w) &= \eps v \bar f_1 + \eps^2 \bar f_2 \qquad && \text{in } \Q \times (0,T) \label{eq:w1}\\
w(0) &= \bar w_0 \qquad && \text{on } \Q,  \label{eq:w2}
\end{alignat}
with $\bar f_1 =  \alpha \dx \bar c \, \bar z - \dx \bar g_u$, $\bar f_2= -\dt \bar g_u$, $\bar w_0= \bar u_0 - \bar g_u(0)$, and homogeneous inflow boundary conditions $w = 0$ on $\Gamma^{in} \times (0,T)$. 
This can be written as an abstract Cauchy problem
\begin{align}
\dt w + \A^\eps w &= f, \qquad  t>0, \label{eq:abs:1}\\
w(0) &= \bar w_0, \label{eq:abs:2}
\end{align}
on the Hilbert space $\H = L^2(\Q)$, with right hand side $f= \frac{1}{\eps} v \bar f_1 + \bar f_2$, and densely defined linear operator $\A^\eps : D(\A^\eps) \subset \H \to \H$, $\A^\eps w =  \frac{1}{\eps} v \dx w + \frac{\sigma}{\eps^2} (w-\bar w)$ with domain 
\begin{align*}
D(\A^\eps)=\{w \in L^2(Q) : v \dx w \in L^2(\Q), \ w|_{\Gamma^{in}}=0\}.
\end{align*}
We note that $D(\A^\eps)\subset\H$ is dense and by \eqref{eq:ibp}, we see that 
\begin{align*}
    \langle \A^\eps w,w\rangle_\Q = \frac{1}{\eps}\langle v\dx w,w\rangle_\Q + \frac{1}{\eps^2}\langle\sigma (w-\bar w),w\rangle_\Q \geq 0,
\end{align*}
for all $w \in D(\A^\eps)$, where $\langle a,b\rangle_\Q=\int_\Q a b \, d(x,v)$; this shows that $-\A^\eps$ is dissipative~\cite[Prop. 2.4.2]{Cazenave_Heraux}. From~\cite[Thm.~3.5]{Agoshkov98}, we deduce that for any $\lambda>0$ and $f \in L^2(\Q)$, the equation $\A^\eps w +\lambda w = f$ has a unique solution $w\in D(\A^\eps)$, and conclude that $-\A^\eps$ is m-dissipative. 
Hence $-\A^\eps$ generates a strongly continuous semigroup of contractions~\cite[Thm.~3.4.4]{Cazenave_Heraux}.
From our assumptions, we further see that $\bar w_0 \in D(\A^\eps)$ and  $f \in W^{1,1}(0,T;\H)$, which implies the existence of a unique classical solution $w \in C^1([0,T];\H) \cap C^0([0,T];D(\A^\eps))$ to \eqref{eq:abs:1}--\eqref{eq:abs:2}; see e.g. \cite[Prop. 4.1.6]{Cazenave_Heraux}. 
%
% \textcolor{purple}{MS: Check this: Don't we need $f\in C^1$? If $f\in W^{1,1}$, then strong solution. Give a precise reference. Given our assumptions $\bar f_1 \in C^1([0,T];L^2(\X))$, $\bar f_2\in H^1(0,T;H^2(\X))\hookrightarrow C^{0,1/2}([0,T];H^2(\X))$}
The function $u = w + \bar g_u$ then is a sufficiently regular solution of \eqref{eq:u:lin} with the required initial and boundary conditions. 
Uniqueness follows from the linearity of the problem. 
\end{proof}
\begin{remark}
Since $-\A^\eps$ generates a semigroup of contractions, we also  obtain the bounds
\begin{align*}
\|w\|_{C([0,T];L^2(Q))} \leq \|\bar w_0\|_{L^2(\Q)} + \|f\|_{L^1(0,T;L^2(\Q))}.
\end{align*}
Inserting $f=\frac{1}{\eps}v\bar f_1 + \bar f_2$, we see that this simple estimate is, however, not uniform in $\eps$. A refined analysis is thus required to obtain sharper a-priori bounds.
\end{remark}

\subsection{A-priori bounds for linearized equations}
As a next step, we derive additional bounds for the solution $(u,c)$ of \eqref{eq:u:lin}--\eqref{eq:c:lin} making explicit the dependence on $\eps$. 
To do so, we assume $z \in C^1([0,T];L^2(\Q))$ is given with $z(0)=\bar u_0$ and we define the weighted norm
\begin{align} \label{eq:z:bound}
\|z\|_{\eps,T}^2 =  \|z\|_{L^\infty(0,T;L^2(\Q))}^2 + \eps^2 \|\dt z\|_{L^\infty(0,T;L^2(\Q))}^2 \le C_z^2.
\end{align}
By careful application of the estimates of Proposition~\ref{pro:c} and additional energy estimates for the linearized kinetic equation \eqref{eq:u:lin}, we obtain the following assertions. 
\begin{proposition} \label{pro:ubound}
Let $z \in C^1([0,T];L^2(\Q))$ be given with $z(0)=\bar u_0$ and bounded by \eqref{eq:z:bound}. 
%
% Then the solution $\bar c$ to \eqref{eq:c:lin} provided by Proposition~\ref{pro:c} satisfies 
% \begin{align}
% \|\bar c\|_{L^2(0,T;H^2(\X))} + \|\bar c\|_{H^1(0,T;L^2(\X)} &\le C \\
% \|\dt \bar c\|_{L^2(0,T;H^2(\X))} + \|\dt \bar c\|_{H^1(0,T;L^2(\X)} &\le C \eps^{-1} 
% \end{align}
% Furthermore, the solution of $u$ of \eqref{eq:u:lin} provided by Proposition~\ref{pro:u} satisfies 
% \begin{align}
% *** 
% \end{align}
% The constant $C$ in these estimates is independent of $\eps$.
Then the solution $(\bar u,\bar c)$ of \eqref{eq:u:lin}--\eqref{eq:c:lin} satisfies the bounds of Theorem~\ref{thm:2} with a constant $C$ that only depends on the bounds for the problem data in Assumption~\ref{ass:1} and $C_z$ in \eqref{eq:z:bound}.
\end{proposition}
\begin{proof}
The estimates (f)--(h) of Theorem~\ref{thm:2} for the function $\bar c$ follow readily from Proposition~\ref{pro:c} .
To derive the improved bounds for $u$, we continue by testing \eqref{eq:w1} with $w/\eps^2$. 
%
%Since $w$ is a classical solution, all steps are well-defined. 
%
Using the integration-by-parts formula~\eqref{eq:ibp}, we see that 
\begin{align*}
\frac{1}{2} \frac{d}{dt} \|w\|^2_{L^2(\Q)} 
&+ \frac{1}{2\eps} \|w\|_{L^2(\Gamma^{out};|v| d\Gamma)}^2 + \frac{1}{\eps^2} \|\sqrt{\sigma}(w - \bar w)\|^2_{L^2(\Q)} \\
&= \frac{1}{\eps}( v \bar f_1,w - \bar w)_{L^2(\Q)} + (\bar f_2, w)_{L^2(\Q)} \\
&\le \frac{1}{2 \sigma_{min}} \|\bar f_1\|^2_{L^2(\X)} + \frac{1}{2 \eps^2} \|\sqrt{\sigma}(w - \bar w)\|^2_{L^2(\Q)} +  \|\bar f_2\|^2_{L^2(\X)} + \frac{1}{2} \|w\|^2_{L^2(\Q)}.
\end{align*}
The second term on the right hand side can be absorbed into the left hand side of this inequality. By application of Grönwall's inequality, we then immediately obtain
\begin{align} \label{eq:w:bound}
\|w\|_{L^\infty(0,T;L^2(\Q))}^2 &+ \frac{1}{\eps} \|w\|_{L^2(0,T;L^2(\Gamma^{out};|v| d\Gamma))}^2 + \frac{\sigma_{min}}{\eps^2} \|w - \bar w\|^2_{L^2(0,T;L^2(\Q))} \\
&\le e^T \Big(\|w(0)\|_{L^2(\Q)}^2 + \frac{1}{\sigma_{min}} \|\bar f_1\|_{L^2(0,T;L^2(\X))}^2 + 2\|\bar f_2\|^2_{L^2(0,T;L^2(\X))}\Big) \notag
\end{align}
From $w(0)=\bar u_0 - \bar g_u(0)$ and our assumptions, we see that $\|\bar w(0)\|_{L^2(\Q)} \le C_0$. 
Using Hölder and Young inequalities, we can further estimate 
\begin{align*}
% \|\bar f_1\|^2_{L^2(L^2)} 
% &\le 2 \alpha^2 \|\dx \bar c\|_{L^2(L^\infty)}^2 \|\bar z\|_{L^\infty(L^2)}^2 + 2 \|\dx \bar g_u\|^2_{L^2(L^2)} \le C_1.
\|\bar f_1\|^2_{L^2(0,T;L^2(\X))} 
&\le 2 \alpha^2 \|\dx \bar c\|_{L^4(0,T;L^\infty(\X))}^2 \|\bar z\|_{L^4(0,T;L^2(\X))}^2 + 2 \|\dx \bar g_u\|^2_{L^2(0,T;L^2(\X))} \\
%&\le  \new{(C_1'+C_1'' C_z^2) T^{1/2} C_z^2 + C_1} \\
&\le C_1 C_z^4 T^{1/2} + C_2.
%  C_1 T^{1/2} C_z^2 + C_1'.
%  C_1 T^{1/2} C_z^2 + C_1'.
\end{align*}
The last bound follows from previous estimates for $\bar c$, our assumptions on the problem data, and the estimate $\|a\|_{L^4(0,T)} \le T^{1/4} \|a\|_{L^\infty(0,T)}$, which again is a consequence of Hölder's inequality.
The definition of $\bar f_2$ implies that $\|\bar f_2\|_{L^2(0,T;L^2(\X))} = \|\dt \bar g_u\|_{L^2(0,T;L^2(\X))} \le C_3$.
Note that the constants $C_i$ can be chosen independently of $\eps$ and $T$.
%
% \kh{The constant $C_1$ also contains powers of $C_z$ and $T$: $\|\dx \bar c\|_{L^4_tL^\infty_x}\leq C_1+C_2T^{1/4}C_z$ by Prop. 4. Should they be shown/mentioned explicitly? Of course $T$ can be estimated by $T^\star$ and no dependence exists anymore. (same below in $\bar f_1'$ estimation and in 3.4 Step 1 equation)}
%
% \herbert{HE: Ja, genau so war das gemeint ... Ich hab nur die "T"s gelassen, die man später für die Kontraktion benötigt. }
%
The estimates (a)--(c) in Theorem~\ref{thm:2} for the function $u = \bar g_u + w$ then follow from the triangle inequality. 

For establishing the remaining bounds, we proceed as follows. By formal differentiation of \eqref{eq:w1}--\eqref{eq:w2}, we see that $w' = \dt w = \dt u - \dt \bar g_u$ is a mild solution of 
\begin{align*}
\eps^2 \dt w' + \eps v \dx w' + \sigma (w' - \bar w') &= \eps v \bar f_1' + \eps^2 \bar f_2', \\
w'(0) &= w'_0,
\end{align*}
with 
$\bar f_1' =  \alpha \dtx \bar c \bar z + \alpha \dx \bar c \dt \bar z - \dtx \bar g_u$, 
$\bar f_2' = -\dtt \bar g_u$,
and $w_0' = \eps^{-1} v (\alpha \dx \bar c_0 \bar u_0 - \dx\bar u_0) - \dt \bar g_u(0)$. 
By our assumptions, we have
$\|\eps w_0'\|_{L^2(\Q)} \le C_0$ and $
\|\bar f_2'\|_{L^2(L^2)}^2 = \|\dtt \bar g_u\|_{L^2(L^2)}^2 \le C_3$. 
By similar arguments as used before, we can further show that
\begin{align*}
\|\bar f_1'\|_{L^2(L^2)}^2 
&\le 3\alpha^2 (\|\dtx \bar c\|_{L^4(L^\infty)}^2 \|\bar z\|_{L^4(L^2)}^2 +  \|\dx \bar c\|_{L^4(L^\infty)}^2 \|\dt \bar z\|_{L^4(L^2)}^2) + 3\|\dtx \bar g_u\|^2_{L^2(L^2)} \\
%&\new{\le (C_1'+C_1''\eps^{-2} + C_1'''T^{1/2}C_z^2\eps^{-2}) T^{1/2} C_z^2  + C_1}  \\
&\leq C_1 C_z^4 \eps^{-2} T^{1/2} + C_2.%\le C_1 T^{1/2} C_z^2 \eps^{-2} + C_1'.
\end{align*}
An application of the energy estimate \eqref{eq:w:bound} to $w'=\dt w$ %, which can be justified through an approximation process, 
then yields
$\|\eps \dt w\|^2_{L^\infty(L^2)} \le C$. 
By rearranging~\eqref{eq:w1} and using the previous bounds, we further obtain $\|v \dx w\|_{L^2(L^2(\Q))} \le C'$. 
The estimates (a)--(b) and (d)--(e) of Theorem~\ref{thm:2} for $u=w + \bar g_u$ then follow by the triangle inequality and the assumptions on the problem data. 
\end{proof}

\subsection{Existence of a unique solution}

As noted in the beginning of Section~\ref{sec:proof:thm:2}, we will use a fixed-point argument to establish existence of a solution to \eqref{eq:u}--\eqref{eq:c:ibc}. 
To this end, we denote by 
$\Phi : \SC_T \to C^1([0,T];L^2(\Q))$, $z \mapsto u$ the mapping that assigns to a given $z$ in
\begin{align}
\SC_T =  \{ z \in C^1([0,T];L^2(\Q)) : \|z\|_{\eps,T} \le C_z, \, z(0) = \bar u_0 \}
\end{align}
the solution $u$ of the linearized equations \eqref{eq:u:lin}--\eqref{eq:c:lin} with the required initial and boundary conditions. 
From Propositions \ref{pro:c} and \ref{pro:u}, we see that this mapping is well-defined. 
We will show now that $\Phi$ is a self-mapping and a contraction on $\SC_T$ whenever $C_z$ is chosen large enough
and $0<T\le T^*$ is chosen small enough. Note that choosing $C_z$ sufficiently large will also ensure that $\SC_T$ is non-empty. 

\subsection*{Step~1}
Let $z \in \SC_T$ and $(u,\bar c)$ denote the solution of \eqref{eq:u:lin}--\eqref{eq:c:lin}. Further set $w=u-\bar g_u$. Then from the estimates derived in the proof of Proposition~\ref{pro:ubound}, we see that 
\begin{align*}
\|u\|_{\eps,T}^2 \le 2\|\bar g_u\|_{\eps,T}^2 + 2\|w\|_{\eps,T}^2  
%&\le  \new{C_1 + (C_2 + C_2'T^{1/2}\|z\|_{\eps,T}^2)T^{1/2} \|z\|_{\eps,T}^2}  \\
&\le C_1 C_z^4 T^{1/2}+ C_2 
%\le C_1 + C_2 T^{1/2} \|z\|_{\eps,T}^2
\end{align*}
with constants $C_1$, $C_2$ that can be chosen independent of $\eps$ and $T$. 
%
% We may choose $C_z \ge \sqrt{C_1+C_2}$. % \sqrt{2 C_1}$. 
% Then for $T \le \min(1/C_z^8, %1/(4 C_2^2),
% T^*)$, we have $\|u\|_{\eps,T}^2 \le C_1+C_2 \le%C_z^2/2 + C_z^2/2 = 
% C_z^2$.
By choosing $C_z$ sufficiently large and $T$ sufficiently small, we obtain $\|u\|_{\eps,T}\leq C_z$.
Moreover, by construction we know that $u \in C^1([0,T];L^2(\Q))$ and it satisfies the required initial and boundary conditions. 
Hence $\Phi$ is a self-mapping on $\SC_T$ for this choice of $T$ and $C_z$.

\subsection*{Step~2}
Let $z_1,z_2 \in \SC_T$ be given and $(u_1,\bar c_1)$, $(u_2,\bar c_2)$ be the corresponding solutions of the system \eqref{eq:u:lin}--\eqref{eq:c:lin} with the required initial and boundary conditions.
Then the two functions $\bar d =  \bar c_1-\bar c_2$ and $w =  u_1 - u_2$ satisfy
\eqref{eq:lin:c1} and \eqref{eq:w1} with 
$\bar f= \gamma (\bar z_1 - \bar z_2)$, $\bar f_1= \alpha (\dx \bar c_1 \bar z_1 - \dx \bar c_2 \bar z_2)$, and $\bar f_2= 0$, and with homogeneous initial and boundary conditions. 
From the estimates of Proposition~\ref{pro:c} and the continuous embeddings of $H^1(\X) \to L^\infty(\X)$ and $L^\infty(0,T) \to L^4(0,T)$, we see that
\begin{align*}
\|\dx \bar d\|_{L^4(0,T;L^\infty(\X))} &\le C \|z_1 - z_2\|_{L^\infty(0,T;L^2(\Q))} \\
\|\dtx \bar d\|_{L^4(0,T;L^\infty(\X))} &\le C \|\dt z_1 - \dt z_2\|_{L^\infty(0,T;L^2(\Q))}
\end{align*}
with a uniform constant $C$ independent of $\eps$ and $T$.
Using the a-priori bounds on solutions of the linearized equations obtained so far, we can then further estimate
\begin{align*}
\|\bar f_1\|^2_{L^2(L^2)} 
&\le 2\alpha^2 (\|\dx \bar c_1 - \dx \bar c_2\|_{L^4(L^\infty)}^2 \|\bar z_1\|_{L^4(L^2)}^2 + \|\dx \bar c_2\|_{L^4(L^\infty)}^2 \|\bar z_1 - \bar z_2\|_{L^4(L^2)}^2) \\
&\le C T^{1/2} \|\bar z_1 - \bar z_2\|_{L^\infty(0,T;L^2)}^2,
\end{align*}
and in a similar manner, we obtain 
\begin{align*}
\|\dt \bar f_1\|^2_{L^2(L^2)}
&\le 4 \alpha^2 (\|\dtx (\bar c_1 - \bar c_2)\|^2_{L^4(L^\infty)} \|z_1\|^2_{L^4(L^2)} + \|\dx \bar c_2\|_{L^4(L^\infty)}^2 \|\dt (z_1-z_2)\|^2_{L^4(L^2)} \\
& \qquad \qquad  + \|\dtx \bar c_2\|^2_{L^4(L^\infty)} \|z_1 - z_2\|_{L^4(L^2)}^2 + \|\dx  (\bar c_1-\bar c_2)\|^2_{L^4(L^\infty)} \|\dt z_1\|^2_{L^4(L^2)} )\\
&\le C T^{1/2} (\|\dt (z_1-z_2)\|_{L^\infty(L^2)} + \eps^{-2} \|z_1 - z_2\|^2_{L^\infty(L^2)}).
\end{align*}
By combination of the estimates derived in the proof of Proposition~\ref{pro:ubound}, we thus obtain 
\begin{align*}
\|u_1-u_2\|_{\eps,T}^2 
\le C e^T T^{1/2} \|z_1 - z_2\|_{\eps,T}^2.
\end{align*}
We note that $C$ is again independent of $\eps$ and $T$. For $T < \min\{1/(C^2 e^{2 T^*}),T^*\}$, we further see that $L =  C e^{T} T^{1/2} < 1$. Hence $\Phi$ is a contraction on $\SC_T$ for this choice of $T$. 

\subsection*{Step~3.}
In summary, we have shown that $\Phi$ is a self-mapping and a contraction on $\SC_T$ whenever $C_z$ is chosen sufficiently large and $T$ sufficiently small. 
For $T>0$, the set $\SC_T$ is nonempty and closed. The existence of a unique fixed-point $u = \Phi(u)$ in $\SC_T$ then follows readily from Banach's fixed-point theorem. 
In addition, any fixed point of $\Phi$ also is a solution of \eqref{eq:u}--\eqref{eq:c:ibc}, and vice versa. 
This yields existence of a unique solution on $[0,T]$.

\subsection{A-priori bounds}
By construction, $(u,\bar c)$ also is a solution of the linearized equations \eqref{eq:u:lin}--\eqref{eq:c:lin} with $z=u$ and satisfying the required initial and boundary conditions. 
The bounds in Theorem~\ref{thm:2} then follow immediately  form the ones of Proposition~\ref{pro:c} and Proposition~\ref{pro:ubound}.
This concludes the proof of Theorem~\ref{thm:2}. \qed

\section{Proof of Theorem~\ref{thm:3}, Part~1}
\label{sec:proof:thm:3a}

We now show that solutions $(u^\eps,\ceps)$ of the kinetic chemotaxis model \eqref{eq:u}--\eqref{eq:c:ibc} converge in an appropriate sense to the unique solution $(\bar u^0,\co)$ of the Keller-Segel system \eqref{eq:u0}--\eqref{eq:c0:ibc}. 
We will use $\langle a,b\rangle_S = \int_S a \, b \, ds$ to denote various $L^2$-scalar products arising in our analysis.

\subsection{Keller-Segel system}
Let us begin with summarizing some properties about solutions to the limit system which are required later on. 
By weak solution, we mean a pair of functions $(\bar u^0,\co)$ such that
\begin{align}
&\bar u^0 \in L^2(0,T;H^1(\X)) \cap H^1(0,T;H^{-1}(\X)) \label{eq:u0:reg}\\
&\co \in L^2(0,T;H^2(\X)) \cap H^1(0,T;L^2(\X)) \label{eq:c0:reg},
\end{align}
which satisfies \eqref{eq:c0} pointwise a.e., \eqref{eq:u0:ibc}--\eqref{eq:c0:ibc} in a trace sense and \eqref{eq:u0} in a weak sense, i.e., 
\begin{align} \label{eq:u0:weak}
\langle \dt \bar u^0,\bar \phi\rangle_\X + \langle \bar \mu \dx \bar u^0 - \bar \chi \dx \bar c^0 \bar u^0, \dx \bar \phi\rangle_\X  = 0
\end{align}
for all $\bar \phi \in H_0^1(\X)$ and for a.a. $t \in (0,T)$. Let us note that the solution components $\bar u^0$, $\co$ here depend on time, while the test function is independent of time. 
\begin{proposition} \label{pro:u0c0}
Let Assumption~\ref{ass:1} be valid. Then \eqref{eq:u0}--\eqref{eq:c0:ibc} has a unique weak solution.
% \begin{align}
% &\bar u^0 \in L^2(0,T;H^1(\X)) \cap H^1(0,T;H^{-1}(\X)) \label{eq:u0:reg}\\
% &\co \in L^2(0,T;H^2(\X) \cap H^1(0,T;L^2(\X)) \label{eq:c0:reg}.
% \end{align}
% By weak solution, we mean a pair of functions $(\bar u^0,\co)$ with the given regularity 
% which satisfies \eqref{eq:c0} in pointwise sense a.e., \eqref{eq:u0:ibc}--\eqref{eq:c0:ibc} in a trace sense and \eqref{eq:u0} in a weak sense, i.e., 
% \begin{align} \label{eq:u0:weak}
% \langle \dt \bar u^0,\bar \phi\rangle_\X + \langle \bar \mu \dx \bar u^0 - \bar \chi \dx \bar c^0 \bar u^0, \dx \bar \phi\rangle_\X  = \langle \gamma \bar u^0,\bar \phi \rangle_\X 
% \end{align}
% for all $\bar \phi \in L^2(0,T;H_0^1(\X))$ and for a.a. $t \in (0,T)$. Let us note that the solution components $\bar u^0$, $\co$ here depend on time, while the test function is independent of time.
%
%\herbert{Check: Moreover, $\co \in L^p(0,T;W^{2,q}(\X) \cap W^{1,p}(0,T;L^q(\X))$ for any $1 \le p,q < \infty$. }
%
If, in addition, $\sigma \in W^{1,\infty}(\X)$% and $\bar u_0 \in H^2(\X)$
, then $\bar \mu,\bar \chi \in W^{1,\infty}(\X)$ and
\begin{align*}
\bar u^0 &\in H^1(0,T;H^1(\X)) \cap L^\infty(0,T;H^2(\X)) .
\end{align*}
\end{proposition}
\begin{proof}
Existence of a unique local weak solution $\bar u^0 \in L^2(0,T^*;H^1(\X)) \cap H^1(0,T^*;H^{-1}(\X))$ and $\co \in L^2(0,T^*;H^2(\X)) \cap H^1(0,T^*;L^2(\X))$ up to some time $T^*>0$ depending only on the problem data has been proven in  \cite{EggerSchoebel20}. 
Further note that the regularity of $\sigma$ carries over to $\bar \mu$, $\bar \chi$ immediately.
The additional regularity of the solution is established by bootstrapping. For completeness, let us briefly sketch the main arguments:
Interpolation estimates~\cite[Lemma 7.8]{Roubicek13} provide $\bar u^0\in L^4(L^{4})$.
Maximal regularity~\cite{Dore93} for \eqref{eq:c0} then yields  $\co \in L^4(W^{2,4})\cap W^{1,4}(L^{4})$, and hence $\co \in C([0,T];W^{3/2,4}(\X))$ 
using interpolation~\cite[Thm. III.4.10.2]{amann1995linear}.
% Using again interpolation according to
% \begin{equation}\label{eq:interpolation}
%     L^p(W^{2,p})\cap W^{1,p}(L^{p}) \subset C([0,T];W^{2(1-1/p),p}(\X)), \quad \text{for $1<p<\infty$,}
% \end{equation}
% this shows $\co\in C([0,T];W^{3/2,4}(\X))$ and thus $\co,\dx\co \in L^\infty(L^\infty)$ by standard embedding theory \cite[Thm. 8.2]{Dinezza12}. 
%
By embedding~\cite[Thm. 8.2]{Dinezza12}, we then obtain $\co,\dx\co \in L^\infty(L^\infty)$. 
From this and previous bounds, we thus conclude that $\dx(\bar\chi \dx\co\bar u^0)\in L^{2}(L^{2})$. Standard parabolic theory then shows $u\in  L^{2}(H^{2}) \cap H^1(L^{2})$; see e.g.~\cite[Thm. 7.1.5]{Evans10}.

Now let $\bar c':= \dt \co$ and $\bar u'= \dt \bar u^0 $ denote the time derivatives of the two solution components. Then by formally differentiating \eqref{eq:c0} in time, we see that
    \[
    \dt \bar  c' -D\dxx \bar c' + \beta \bar c' = \gamma \bar u'
    \]
with source term $\gamma\bar u' \in L^2(L^2)$, initial data $\bar c'(0) = D\dxx \bar c_0 -\beta \bar c_0 + \gamma \bar u_0 \in H^2(\X)$ and boundary data $\bar c' = \dt \bar g_c\in W^{1,\infty}(0,T;\RR^2)$.
From \cite[Thm. 7.1.5]{Evans10} we conclude that $\bar c'\in L^2(H^2)\cap H^1(L^2)$.
Differentiation of 
\eqref{eq:u0} %and \eqref{eq:u0:ibc} 
further yields
\begin{align*}
\dt \bar u' - \dx (\bar \mu \dx \bar u' - \bar \chi \dx \co \bar u') &= \bar f \qquad  \text{in } \X \times (0,T)
\end{align*}
with $\bar f= -\dx (\bar \chi \dtx \co \bar u^0)$.  
From the previous estimates, we infer that $\bar f \in L^2(0,T;L^2(\X))$.
By formally differentiating \eqref{eq:u0:ibc} and using \eqref{eq:u0}, we obtain
\begin{align*}
\bar u' = \bar g_u' \quad \text{on } \partial\X \times (0,T)
\qquad \text{and} \quad 
\bar u'(0) = \bar u_0' \quad \text{on } \X,
\end{align*}
with data $\bar g_u' =  \dt \bar g_u \in W^{1,\infty}(0,T;\RR^2)$ and $\bar u_0' =  \dx (\bar \mu \dx \bar u_0 - \bar \chi \dx \bar c_0 \bar u_0)\in L^2(\X)$. 
%
%L^2(0,T;H^{-1}(\X))$: %$\bar u_0' \in L^2(\X)$ and $\bar g_u' \in H^1(0,T;H^1(\partial\X))$. 
%
By parabolic regularity, we thus conclude that $\bar u' \in L^2(0,T;H^1(\X)) \cap L^\infty(0,T;L^2(\X))$, which yields $\bar u^0 \in H^1(0,T;H^1(\X))$ and $\dt \bar u^0 \in L^\infty(0,T;L^2(\X))$. 
From equation~\eqref{eq:u0} and the regularity of $\bar \mu$, we finally obtain $\dxx \bar u^0 \in L^\infty(0,T;L^2(\X))$. 
\end{proof}
%\kh{Higher regularity on $u$ would require higher initial data regularity: e.g. application of parabolic reg to $u'=\dt u$ with $f\in L^2(L^2)$ would require $u'(0)\in H^1$}

\begin{remark} \label{rem:weak}
For the first component, the above notion of a weak solution is equivalent to requiring 
$\bar u^0 \in L^2(0,T;H^1(\X)) \cap L^\infty(0,T;L^2(\X))$ 
and validity of the variational identity 
\begin{align} \label{eq:u0:weak2}
-\langle \bar u^0, \dt \tilde \psi \rangle_{\X_T} + \langle \bar \mu \dx \bar u^0 - \bar \chi \dx \co \bar u^0, \dx \tilde \psi\rangle_{\X_T} &= \langle \bar u_0, \tilde \psi(0)\rangle_\X
\end{align}
for all $\tilde \psi \in C^\infty([0,T];H_0^1(\X))$ with $\tilde \psi(T)=0$; see \cite{Roubicek13} for instance.  
We will make implicit use of this equivalent definition of a weak solution later on in our proofs.
\end{remark}

\subsection{Convergent subsequences}
Due to the uniform bounds provided in Theorem~\ref{thm:2}, we can extract a sequence $(u^\eps,\ceps)$ of solutions which converges to an appropriate limit.

\begin{proposition} \label{pro:limit}
Under the assumptions of Theorem~\ref{thm:2}, there exist functions
\begin{align*}
\bar u^0 \in L^\infty(0,T;H^1(\X)) \quad \text{and} \quad 
\co \in L^2(0,T;H^2(\X)) \cap H^1(0,T;L^2(\X)) 
\end{align*}
and a sequence $(u^\eps,\ceps)_{\eps>0}$ of solutions to \eqref{eq:u}--\eqref{eq:c:ibc} such that with $\eps \to 0$ one has
\begin{alignat*}{5}
u^\eps &\rightharpoonup \bar u^0 \qquad \qquad 
&& \text{weakly in } &&L^2(0,T;L^2(\Q)) \\
v \dx u^\eps &\rightharpoonup v \dx \bar u^0 
&& \text{weakly in } &&L^2(0,T;L^2(\Q)) \\
\ceps &\rightharpoonup \co 
&& \text{weakly in } &&L^2(0,T;H^2(\X)) \cap H^1(0,T;L^2(\X)) \\
\ceps &\to \co 
&& \text{strongly in } &&L^2(0,T;H^1(\X)).
\end{alignat*}
Moreover, the limit function $\co$ satisfies the
initial and boundary conditions stated in \eqref{eq:c0:ibc}.
\end{proposition}
\begin{proof}
The assertions about $\ceps$ and $\co$ follow immediately from the estimates of Theorem~\ref{thm:2}, the Banach-Alaoglou Theorem~\cite{Conway19}, the Aubin-Lions compactness lemma~\cite[Lemma 7.7]{Roubicek13}, and the continuity of the trace operators in space and time. 
By the uniform bounds for $u^\eps$, we further obtain a subsequence $u^\eps$ and a limit $u^0 \in L^\infty(0,T;L^2(\Q))$ such that $u^\eps \rightharpoonup u^0$ and $v \dx u^\eps \rightharpoonup v \dx u^0$ weakly in $L^2(0,T;L^2(\Q))$. 
From the weak lower-semicontinuity of norms and the estimates of Theorem~\ref{thm:2}, we further deduce that 
\begin{align*}
\|u^0 - \bar u^0\|_{L^2(0,T;L^2(\Q))} 
\le \liminf_{\eps \to 0} \|u^\eps - \bar u^\eps\|_{L^2(0,T;L^2(\Q))} = 0,
\end{align*}
which implies $u^0 = \bar u^0 \in L^2(0,T;L^2(\X))$. By uniqueness of the weak limit, we further see that $v \dx u^0 = v \dx \bar u^0 \in L^2(0,T;L^2(\Q))$ which implies $u^0 = \bar u^0 \in L^2(0,T;H^1(\X))$. 
\end{proof}
As part of our analysis in the next subsection, we will show that the limit $\bar u^0$ also has a weak time derivative and satisfies the initial and boundary conditions \eqref{eq:u0:ibc}. 

\subsection{Further properties of the limiting functions}
To conclude the proof of the first part of Theorem~\ref{thm:3}, we need to show that the limit functions $(\bar u^0,\co)$ provided by Proposition~\ref{pro:limit} are indeed a weak solution to the Keller-Segel system \eqref{eq:u0}--\eqref{eq:c0:ibc}.

\subsection*{Step~1.}
From the linearity of the equations \eqref{eq:c} and \eqref{eq:c0}, and the assertions of Proposition~\ref{pro:limit}, we can immediately deduce that $(\bar u^0,\co)$ satisfies \eqref{eq:c0} and \eqref{eq:c0:ibc}.

\subsection*{Step~2.}
For the analysis of the kinetic equation, we need the following auxiliary result. 
Let $\tilde \psi \in C^\infty([0,T];H_0^1(\X))$ with $\tilde \psi(T)=0$ be given. Then, from \eqref{eq:u} and \eqref{eq:u:ibc}, 
\begin{align*}
\langle v \dx u^\eps, \sigma^{-1} v \dx \tilde \psi\rangle_{\Q_T} 
&= \langle \eps u^\eps, \sigma^{-1} v\dtx \tilde \psi\rangle_{\Q_T} + \langle \alpha v \dx \ceps \bar u^\eps - \eps^{-1}\sigma (u^\eps - \bar u^\eps), \sigma^{-1} v \dx \tilde \psi\rangle_{\Q_T}. 
\end{align*}
For abbreviation we use the notation $\Q_T= \Q\times (0,T)$ and $\X_T= \X\times (0,T)$ in the following. By the uniform bounds for $u^\eps$, we see that 
$\langle \eps u^\eps, \sigma^{-1} v \dtx \tilde\psi\rangle_{\Q_T} \to 0$
with $\eps \to 0$.  
The first part in the last integral can be split into
\begin{align*}
\langle \alpha v \dx \ceps \bar u^\eps,\sigma^{-1} v \dx \tilde \psi\rangle_{\Q_T}
= \langle \alpha v \dx \co \bar u^\eps,\sigma^{-1} v \dx \tilde \psi\rangle_{\Q_T}
+  \langle \alpha v \dx (\ceps - \co) \bar u^\eps,\sigma^{-1} v \dx \tilde \psi\rangle_{\Q_T}.
\end{align*}
The first term on the right hand side converges to $\langle \alpha v \dx \co \bar u^0, \sigma^{-1} v \dx \tilde \psi\rangle_{\Q_T}$. 
Since $\ceps \to \co$ strongly in $L^2(0,T;H^1(\X))$, the last term converges to zero with $\eps \to 0$.
For the remaining term in the identity above, we use \eqref{eq:ibp} and \eqref{eq:u} to obtain that
\begin{align*}
-\langle \eps^{-1}\sigma (u^\eps &- \bar u^\eps), \sigma^{-1} v \dx \tilde \psi\rangle_{\Q_T}
=-\langle \eps^{-1} u^\eps ,v \dx \tilde \psi\rangle_{\Q_T}
= \langle \eps^{-1} v \dx u^\eps, \tilde \psi\rangle_{\Q_T} \\
&= -\langle \dt u^\eps, \tilde \psi \rangle_{\Q_T} - \langle \eps^{-2} \sigma (u^\eps - \bar u^\eps), \tilde \psi \rangle_{\Q_T} + \langle \eps^{-1} \alpha v \dx \ceps \bar u^\eps ,\tilde \psi\rangle_{\Q_T} \\
&= \langle u^\eps, \dt \tilde \psi\rangle_{\Q_T}  + \langle u^\eps(0),\tilde \psi(0)\rangle_{\Q}.
\end{align*}
By combination of these results and those of Proposition~\ref{pro:limit}, we conclude that 
\begin{align*}
%&\langle \dx \bar u^0, \bar \mu \dx \tilde \psi\rangle_{\X_T} = 
\langle v \dx \bar u^0, \sigma^{-1} v \dx \tilde \psi\rangle_{\Q_T}  
&= \lim_{\eps \to 0} \langle v \dx u^\eps, \sigma^{-1} v \dx \tilde \psi\rangle_{\Q_T}\\
&= \langle \alpha v \dx \co \bar u^0, \sigma^{-1} v \dx \tilde \psi\rangle_{\Q_T} + \langle \bar u^0, \dt \tilde \psi\rangle_{\Q_T} + \langle \bar u_0, \tilde \psi(0)\rangle_{\Q}.
\end{align*}

\subsection*{Step~3.}
%
%We now show that the limit $(\bar u^0,\co)$ provided by Proposition~\ref{pro:limit} satisfies \eqref{eq:u0:weak2}.  
%
By definition of $\bar \chi$ and $\bar \mu$, the previous results, 
%and noting that $\int_\Q \bar w(x) \, d(x,v) = 2 \int_\X \bar w(x) \, dx$, 
further yield that 
\begin{align*}
2 \langle \bar \mu \dx \bar u^0, \dx \tilde \psi\rangle_{\X_T}
&= \langle v \dx \bar u^0, \sigma^{-1} v \dx \tilde \psi\rangle_{\Q_T} \\
&= 2 \langle \bar \chi \dx \co \bar u^0, \dx \tilde \psi \rangle_{\X_T} +  2\langle \bar u^0, \dt \tilde \psi\rangle_{\X_T} + 2\langle \bar u_0, \tilde \psi(0)\rangle_{\X}.
\end{align*}
This shows that the limit $(\bar u^0,\co)$ provided by Proposition~\ref{pro:limit} satisfies \eqref{eq:u0:weak2}. 
In particular,
\begin{align*}
\langle \dt \bar u^0,\tilde \psi\rangle_{\X_T} 
&= -\langle \bar \mu \dx \bar u^0 - \bar \chi \dx \co \bar u^0, \dx \tilde \psi\rangle_{\X_T}
\end{align*}
for all $\tilde \psi \in C_0^\infty((0,T);H_0^1(\X))$. From the regularity of $\bar u^0$ and $\co$, we infer that
\begin{align*}
\bar \mu \dx \bar u^0 - \bar \chi \dx \co \bar u^0 \in L^2(0,T;L^2(\X)).
\end{align*}
This shows that $\bar u^0$ has a weak time derivative $\dt \bar u^0 \in L^2(0,T;H^{-1}(\X))$ and that the variational identity \eqref{eq:u0:weak} holds for a.a. $t \in (0,T)$.
\subsection*{Step~4.}
By combination of \eqref{eq:u0:weak} and \eqref{eq:u0:weak2}, we can immediately see that $\bar u^0(0)=\bar u_0$. 
Now let $\tr : u \mapsto u|_{\Gamma}$ denote the trace operator for $u \in U=\{v \in L^2(\Q) : v \dx u \in L^2(\Q)\}$. 
By use of the integation-by-parts formula \eqref{eq:ibp}, we obtain
\begin{align}
\|\tr\, u\|_{L^2(\Gamma;|v| d\Gamma)}^2 = \int_\Gamma |\tr\, u|^2 |v| \, d\Gamma = 2 \int_\Q v \dx u \, u \, d(x,v) \le 2 \|v \dx u\|_{L^2(\Q)} \|u\|_{L^2(\Q)}.
\end{align}
This shows that the trace operator $\tr : L^2(0,T;U) \to L^2(0,T;L^2(\Gamma;|v| d\Gamma))$ is continuous as a mapping between these spaces. 
By Proposition~\ref{pro:limit}, we know that $\bar u^\eps \rightharpoonup \bar u^0$ weakly in $L^2(0,T;U)$, and hence $\tr\, u^\eps \rightharpoonup \tr\, \bar u^0$ weakly in $L^2(0,T;L^2(\Gamma;|v| d\Gamma))$. 
Since $\tr\, u^\eps=\bar g_u$~on~$\Gamma^{in}\times(0,T)$, we infer that $\tr\, \bar u^0=\bar g_u$ on $\Gamma^{in}\times(0,T)$, and since $\bar u^0$ is independent of $v$, we further conclude that $\bar u^0=\bar g^u$ on $\partial \X\times(0,T)$ in a trace sense. 
In summary, we thus have shown that $\bar u^0$ satisfies the initial and boundary conditions \eqref{eq:u0:ibc}, and hence $(\bar u^0,\co)$ is a weak solution of \eqref{eq:u0}--\eqref{eq:c0:ibc}. 

\subsection*{Step~5.}
Since every sequence of solutions $(u^\eps,\ceps)$ to \eqref{eq:u}--\eqref{eq:c:ibc} for $\eps \to 0$ has a subsequence which converges weakly to a weak solution $(\bar u^0,\co)$ of \eqref{eq:u0}--\eqref{eq:c0:ibc} and since the weak solution to this problem is unique, we obtain weak convergence with $\eps \to 0$ for any sequence $(u^\eps,\ceps)$ of solutions. 
This concludes the first part of the proof of Theorem~\ref{thm:3}.

\section{Proof of Theorem~\ref{thm:3}, Part~2}
\label{sec:proof:thm:3b}
In this section, we establish the quantitative convergence rates announced in Theorem~\ref{thm:3}.
Following standard arguments \cite{Dautray93,EggerSchlottbom14}, we decompose
\begin{align}
u^\eps &= \bar u^0 + \eps v \bar u^1 + \phi^\eps
\quad\text{with}\quad \bar u^1 = -\sigma^{-1} \dx \bar u^0  +\alpha\sigma^{-1}  \dx \co \bar u^0, \label{convrate:phiu1:def}\\
\ceps &= \co + \bar\eta^\eps.\label{convrate:eta:def}
\end{align}
To conclude the proof of the theorem, we need to estimate $\bar u^1$ and the two remainder terms $\phi^\eps$ and $\bar \eta^\eps$ defined by the previous expressions. This will be done in the sequel. 

\subsection*{Step~1.}
From the regularity of the limit solution $(\bar u^0,\co)$ provided by Proposition~\ref{pro:u0c0}, we deduce that $\|\bar u^1\|_{L^\infty(0,T;L^\infty(\X))}$, $\|\dx \bar u^1\|_{L^2(0,T;L^2(\X))}$, and $\|\dt \bar u^1\|_{L^2(0,T;L^2(\X))}$ are uniformly bounded.
In addition, we have
\begin{equation}\label{eq:barvu1}
    \overline{v \bar u^1} = 0, \qquad \text{ and consequently} \qquad \bar \phi^\eps = \bar u^\eps - \bar u^0.
\end{equation}
As a consequence, some of the terms in the following investigations will vanish.

\subsection*{Step~2.}
Using the equations defining $\ceps$ and $\co$, we see that $\bar \eta^\eps = \ceps - \co$ satisfies  
\begin{align}
    \dt\bar\eta^\eps - D\dxx\bar\eta^\eps + \beta\bar\eta^\eps
    &= \gamma\bar\phi^\eps \qquad \text{in } \X\times (0,T), \label{eq:eta}
\end{align}
together with homogeneous initial and boundary conditions. 
From standard a-priori estimates for parabolic equations~\cite{Evans10}, we thus obtain 
\begin{align*}
\|\bar \eta^\eps\|_{L^\infty(0,T;L^2(\X))} 
\le \|\bar \eta^\eps\|_{L^\infty(0,T;H^1(\X))} + \|\bar \eta^\eps\|_{L^2(0,T;H^2(\X))} 
\le C \|\phi^\eps\|_{L^2(0,T;L^2(\Q))}.
\end{align*}
% This in particular implies that  *** some estimates ***.  
%$\|\dx \bar\eta^\eps\|_{L^2(0,T;L^\infty)} \le C'\|\phi^\eps\|_{L^2(0,T;L^2(\Q))}$ and also $\|\bar \eta^\eps\|_{L^2(0,T;L^2}\|\phi^\eps\|_{L^2(0,T;L^2(\Q))}$
%
It therefore remains to establish the appropriate bounds for the second remainder $\phi^\eps$.

\subsection*{Step~3.}
From the equations defining $u^\eps$, $\bar u^0$, and $\bar u^1$, we can see that $\phi^\eps$ solves
\begin{alignat}{5}
\eps^2 \dt \phi^\eps + \eps v \dx \phi^\eps + \sigma (\phi^\eps - \bar \phi^\eps) &= \eps f^\eps \qquad && \text{in } \Q \times (0,T), \label{eq:phi1}\\
\phi^\eps &= g^\eps \qquad && \text{on } \Gamma^{in} \times (0,T), \label{eq:phi2} \\
\phi^\eps(0) &= \phi_0^\eps \qquad && \text{on } \Q \label{eq:phi3}
\end{alignat}
with 
$f^\eps=\alpha v \dx \ceps \bar u^\eps - \eps \dt \bar u^0 -  v \dx \bar u^0 - \eps^2 v \dt \bar u^1 - \eps v^2 \dx \bar u^1 - \sigma v \bar u^1$, 
$g^\eps=-\eps v \bar u^1$, and initial value $\phi^\eps_0 = -\eps v \bar u^1(0)$.
Following the arguments in \cite{EggerSchlottbom14_Lpestimates}, 
we split $\phi^\eps = \phi^\eps_f + \phi^\eps_g$ with $\phi^\eps_f$ solving the same system, but with homogeneous boundary conditions, and $\phi^\eps_g$ satisfying the equations with homogeneous right hand side and homogeneous initial conditions. 

\subsection*{Step~3a.}
As a preliminary result, we establish $L^p$ bounds for $\phi^\eps_g$, elaborating explicitly the dependence on $\eps$. 
To do so, we multiply \eqref{eq:phi1} with $f^\eps=0$ by $|\phi^\eps_g|^{p-2} \phi^\eps_g$, which yields 
\begin{align*}
    \eps^2 \tfrac{1}{p} \dt |\phi^\eps_g|^p + \eps v \tfrac{1}{p} \dx |\phi^\eps_g|^p  =  \sigma(\bar \phi^\eps_g-\phi^\eps_g) |\phi^\eps_g|^{p-2}\phi^\eps_g.
\end{align*}
Integration over $\Q\times(0,t)$, application of the integration-by-parts formula~\eqref{eq:ibp}, and of the inhomogeneous boundary and homogeneous initial conditions for $\phi^\eps_g$ yields
\begin{align}\label{eq:phig:1}
    &\tfrac{\eps^2}{p} \|\phi^\eps_g(t)\|^p_{L^p(\Q)} + \int_0^t\int_{\Gamma^{out}} \tfrac{\eps}{p} |v||\phi^\eps_g|^p \ d\Gamma\ ds\\
    &= %\tfrac{\eps^2}{p} \|\phi^\eps_g (0)\|^p_{L^p(\Q)} \! + \! 
    %\int_0^t\int_{\partial \Q^{in}} \tfrac{\eps}{p} |v||\phi^\eps_g|^p \ d(x,v) \ ds \! 
    \int_0^t\int_{\Gamma^{in}} \tfrac{\eps}{p} |v||\eps v \bar u^1|^p \ d\Gamma \ ds
    +  \int_0^t \int_{\Q} \sigma (\bar \phi^\eps_g - \phi^\eps_g) |\phi^\eps_g|^{p-2}\phi^\eps_g \ d(x,v) \ ds.\nonumber%\\
    %&= (i) + (ii) + (iii).\nonumber
\end{align}
By Hölder's inequality, we see that 
\begin{align*}
\int_\V  \bar u |u|^{p-2} u \ dv 
&\leq \left(\int_\V  |\bar u|^p \ dv\right)^{1/p}\left(\int_\V  (|u|^{p-1})^{\tfrac{p}{p-1}} \ dv\right)^{(p-1)/p}
\end{align*}
for any $u \in L^p(\V)$, and Jensen's inequality further yields $\left(\int_\V  |\bar u  |^p \ dv\right)^{1/p} \leq \left(\int_\V  |u|^p \ dv\right)^{1/p}$ for any $p \ge 2$.
By combination of these estimates, we conclude that the last term in the second line of \eqref{eq:phig:1} is non-positive. 
By simplification of the inequality, we thus obtain
\begin{align*}
 \|\phi^\eps_g(t)\|_{L^p(\Q)}  \leq \eps^{-1/p}\left( \int_0^t\int_{\Gamma^{in}}  |v|^2|\eps \bar u^1|^p \ d\Gamma \ ds  \right)^{1/p} \le C \eps^{1-1/p}.
\end{align*}
For the last step, we used the uniform boundedness of $\bar u^1$. 
By letting $p\to \infty$, we obtain 
\begin{align}\label{phig:order}
\|\phi^\eps_g\|_{L^\infty(0,T;L^\infty(\Q))} 
= \lim_{p \to \infty} \|\phi^\eps_g\|_{L^\infty(0,T;L^p(\Q))} \leq C' \eps. 
\end{align}
Since $\X$ is bounded, we further obtain $\|\phi^\eps_g\|_{L^\infty(0,T;L^2(\Q))} \leq C \|\phi^\eps_g\|_{L^\infty(0,T;L^\infty(\Q))} \le C'' \eps$, which is the desired bound for the first component of $\phi^\eps$.

\subsection*{Step~3b.}
We start with having a closer look at the right hand side of \eqref{eq:phi1}. 
From the definition of $\bar u^1$, we see that $\sigma \bar u^1 = \alpha\dx \co \bar u^0 - \dx \bar u^0$. 
Further using \eqref{eq:c0}, we can rewrite 
\begin{align} \label{eq:feps}
f^\eps &=\alpha v (\dx \ceps \bar u^\eps - \dx \co \bar u^0) -\eps^2 v \dt \bar u^1 - \eps [v^2 \dx \bar u^1 + \dx (\bar \mu \dx \bar u^0 - \bar \chi \dx \co \bar u^0)] .
\end{align}
Using the formulas defining $\bar \mu$, $\bar \chi$, and $\bar u^1$, we see that 
\begin{align*}
\int_\V v^2 \dx \bar u^1 + \dx (\bar \mu \dx \bar u^0 - \bar \chi \dx \co \bar u^0) \, dv 
= 0,
\end{align*}
and hence $\bar f^\eps = 0$. 
Let us note that the function $f^\eps$ can be treated like the term $v \bar f_1$ in the system \eqref{eq:w1}--\eqref{eq:w2}. From the energy estimate \eqref{eq:w:bound}, we may thus conclude that 
\begin{align*}
\|\phi^\eps_f\|^2_{L^\infty(0,t;L^2(\Q))} 
\le e^{t} (\|\phi_f^\eps(0)\|^2_{L^2(\Q)} + \|f^\eps\|_{L^2(0,t;L^2(\Q))}^2)
\end{align*}
for any $0 \le t \le T$. 
From the bounds for $\bar u_1$, we obtain $\|\phi^\eps_f(0)\|^2_{L^2(\Q)} \le C_0 \eps^2$. 
The right hand side \eqref{eq:feps} may be split into $f^\eps = f^\eps_0 + \eps f^\eps_1 + \eps^2 f^\eps_2$. 
From the bounds for $\bar u^0$ and $\bar u^1$, we immediately obtain that $\|\eps f^\eps_1 + \eps^2 f^\eps_2\|_{L^2(0,T;L^2(\Q))} \le C \eps$.
Using Hölder and triangle inequalities, the remaining term can be estimated as follows
\begin{align*}
\|f_0^\eps\|_{L^2(L^2)} 
&\le \alpha \|\dx \ceps \bar u^\eps - \dx \co \bar u^0\|_{L^2(L^2)}\\
& \le \alpha (\|\dx \ceps - \dx \co\|_{L^2(L^\infty)} \|\bar u^\eps\|_{L^\infty(L^2)} + \|\dx \co\|_{L^\infty(L^\infty)} \|\bar u^\eps - \bar u^0\|_{L^2(L^2)}).
\end{align*} 
By previous a-priori estimates, we can bound $\|\bar u^\eps\|_{L^\infty(L^2)} \le C$ and $\|\dx \co\|_{L^\infty(L^\infty)} \le C$. 
Using the equation defining $\bar \eta^\eps=\ceps-\co$, the definition of $\phi^\eps$, and the left identity of \eqref{eq:barvu1}, we see that the two terms involving differences can be further bounded by 
\begin{align*}
\|\dx \ceps - \dx \co\|_{L^2(L^\infty)}^2 + \|\bar u^\eps - \bar u^0\|_{L^2(L^2)}^2
\le C_2 \|\phi^\eps\|_{L^2(L^2)}^2 \le C_3 (\|\phi^\eps_f\|^2_{L^2(L^2)} + \eps^2). 
\end{align*}
In the last step, we used the bound for $\|\phi_g^\eps\|^2_{L^2(L^2)}$ from before.
In summary, we thus get 
\begin{align*}
\|f^\eps\|_{L^2(0,t;L^2(\Q))}^2 \le C_1 \|\phi^\eps_f\|^2_{L^2(0,t;L^2(\Q))} + C_2 \eps^2. 
\end{align*}
Inserting this into the energy estimate \label{eq:phif:bound} for $\phi^\eps_f$ and applying a Grönwall inequality then leads to the uniform bound $\|\phi^\eps_f\|^2_{L^\infty(0,T;L^2(\Q))} \le C_2' \eps^2$. 

\subsection*{Step~3c.}
By combination of the previous estimates, we see that 
$$
\|\phi^\eps\|_{L^\infty(L^2)} \le \|\phi^\eps_f\|_{L^\infty(L^2)} + \|\phi^\eps_g\|_{L^\infty(L^2)} \le C \eps.
$$
From the definition of $\phi^\eps$ and the uniform bounds for $\bar u^1$, we then get
$$
\|u^\eps - \bar u^0\|_{L^\infty(0,T;L^2(\Q))} 
\le \eps \|v \bar u^1\|_{L^\infty(0,T;L^2(\Q))} + \|\phi^\eps\|_{L^\infty(0,T;L^2(\Q))} 
\le C \eps.
$$ 
From the arguments of Step~2, we conclude that $\|\ceps - \co\|_{L^\infty(0,T;L^2(\Q))} \le C \eps$ as well, which proves the convergence rates stated in Theorem~\ref{thm:3}. \qed

\section{Discussion}
\label{sec:conclusion}

We established well-posedness of a kinetic chemotaxis model in slab geometry and convergence to a drift-diffusion system of Keller-Segel type in the asymptotic limit. The results are valid on a finite time interval $[0,T]$ with $T>0$ depending on the problem data. 

From the global well-posedness of the limit problem in one space dimension, one would expect that the results should be valid for $T=\infty$. A different analysis would, however, be required, probably involving the positivity of solutions and uniform bounds in $L^1$; compare with the proofs in~\cite{EggerSchoebel20} for the corresponding limit problem.   
Another way to obtain global-in-time results is to replace the nonlinear coupling term in~\eqref{eq:u} by $\eps \alpha v \Psi(\dx \cepsg) \uepsgbar $ for a suitable bounded function $\Psi$, which acts as \emph{flux limiting}. Our results can be generalized to this case.
Such models were recently introduced to avoid nonphysical blow-up behavior and were shown to attain global solutions \cite{Perthame19}; also see \cite{burger06, bellomo2017} for the related analysis of the flux-limited Keller-Segel model in multiple space dimensions.

Let us note that our results can be  extended quite naturally to related chemotaxis models on networks which were introduced as mathematical models for dermal wound healing \cite{Bretti14} or the growth of slime molds \cite{BorscheGoettlich14}. 
The fact that the velocity space $\V=(-1,1)$ involves a whole interval here means that agents can actually move at different speeds. Such systems have been studied numerically in~\cite{Borsche16}.
The asymptotic analysis for somewhat simplified coupling conditions has been derived~\cite{Philippi23}, and corresponding results for the limiting Keller-Segel model on networks were obtained in~\cite{BorscheGoettlich14,EggerSchoebel20}. 

Possible topics for future research include the  extension of our results to more complex interaction mechanisms, involving multiple species, different chemo-attractants, or more complex coupling mechanisms; see \cite{Bitsouni2018,Emako2015,Ren2020} for some recent work in this direction.
The design of asymptotic preserving numerical schemes for chemotaxis in slab geometry and related models on networks, in a similar spirit to \cite{Gosse13,Natalini2012, CarilloYan13}, would be a topic for future research.

\subsection*{Funding}
KH acknowledges support of the German Academic Scholarship Foundation (Studienstiftung des deutschen Volkes) as well as the Marianne-Plehn-Programm.
HE and NP are grateful for support by the German Research Foundation (DFG) through the Collaborative Research Center TRR~154.

%% BIBLIOGRAPHY %%%%%%%%
%\begingroup
\bibliographystyle{abbrv}
\bibliography{biblio}
%\endgroup

\end{document}

%% file: abbreviations.tex
\def\eps{\varepsilon}

\def\V{\mathcal{V}}

\def\Q{\mathcal{Q}}

\def\ceps{\bar c^{\,\eps}}

\def\cepsg{\bar c^{\,\eps}}

\def\uepsgbar{ \bar u^{\eps}}

\def\uepsg{ u^{\eps}}
\def\co{\bar c^{\,0}}
\def\cog{\bar c^{\,0}}

\def\uogbar{\bar u^0}

\def\dx{\partial_x}
\def\dxx{\partial_{xx}}
\def\dt{\partial_t}
\def\dtt{\partial_{tt}}
\def\dtx{\partial_{tx}}

\newcommand{\tr}{{\mathrm{tr}}}

\def\RR{\mathbb{R}}

\def\A{\mathcal{A}}

\def\X{\mathcal{X}}
\def\H{\mathcal{H}}

\def\SC{\mathcal{S}}

\usepackage{accents}

%% file: domain.tikz
\begin{tikzpicture}[scale = 2.3, label distance = 2mm]
	\draw[white,fill=black!5] (-1.5,0.5) rectangle (1.5,-0.5);
    \coordinate[label=left: $0$] (A) at (-1.5,0);
    \coordinate[label=right: $\ell$] (B) at (1.5,0);
    \coordinate[label=left: $1$] (C) at (-1.5,0.5);
    \coordinate[label=left: $-1$] (D) at (-1.5,-0.5);
    \coordinate (E) at (1.5,0.5);
    \coordinate (F) at (1.5,-0.5);
    % \coordinate[label=right: $1$] (E) at (3,0.5);
    % \coordinate[label=right: $-1$] (F) at (3,-0.5);
    %
    \draw[-, line width=4pt, red] (A) to (C);
    \draw[densely dotted, line width=4pt, blue] (A) to (D);
    \draw[densely dotted, line width=4pt, blue] (B) to (E);
    \draw[-, line width=4pt, red] (B) to (F);
    \draw[-,thick] (C) to (E);
    \draw[-,thick] (D) to (F);
    \draw[-latex, very thick] (A) to (B);
	\draw[->,thick] (-0.25,0.4) -- (0.25,0.4);	
	\draw[->,thick] (-0.175,0.3) -- (0.175,0.3);	
	\draw[->,thick] (-0.1,0.2) -- (0.1,0.2);	
	\draw[->,thick] (-0.025,0.1) -- (0.025,0.1);	
	\draw[<-,thick] (-0.25,-0.4) -- (0.25,-0.4);	
	\draw[<-,thick] (-0.175,-0.3) -- (0.175,-0.3);	
	\draw[<-,thick] (-0.1,-0.2) -- (0.1,-0.2);	
	\draw[<-,thick] (-0.025,-0.1) -- (0.025,-0.1);	
\end{tikzpicture}